\documentclass{amsart}

\usepackage[foot]{amsaddr}
\let\oldtocsection=\tocsection

\let\oldtocsubsection=\tocsubsection

\let\oldtocsubsubsection=\tocsubsubsection

\renewcommand{\tocsection}[2]{\vspace{0.5em}\hspace{0em}\oldtocsection{#1}{#2}}
\renewcommand{\tocsubsection}[2]{\vspace{0.5em}\hspace{1em}\oldtocsubsection{#1}{#2}}
\renewcommand{\tocsubsubsection}[2]{\vspace{0.5em}\hspace{2em}\oldtocsubsubsection{#1}{#2}}
\usepackage{graphicx,cancel,xcolor,hyperref,comment,graphicx,geometry}

\usepackage{tikz}
\usetikzlibrary{decorations.pathreplacing}
\usepackage{graphicx,url,etoolbox}
\usepackage{lipsum}
\usepackage{dsfont}
\usepackage{amsmath,amssymb}

\usepackage{verbatim}

\usepackage{fancyhdr}
\usepackage{lastpage}

\def\nline{\\ \noalign{\medskip}}

\newtheorem{theoreme}{Theorem}[section]

\theoremstyle{definition}

\numberwithin{equation}{section}

\renewenvironment{proof}{{\bfseries \noindent Proof.}}{\demo}
\newcommand\xqed[1]{%
	\leavevmode\unskip\penalty9999 \hbox{}\nobreak\hfill
	\quad\hbox{#1}}
\newcommand\demo{\xqed{$\square$}}

\hypersetup{bookmarks, colorlinks, urlcolor=blue, citecolor=blue, linkcolor=blue, hyperfigures, pagebackref,
	pdfcreator=LaTeX, breaklinks=true, pdfpagelayout=SinglePage, bookmarksopen=true,bookmarksopenlevel=2}

\usepackage{comment}

\def\R{\mathbb R}

\def\N{\mathbb N}

\def\C{\mathbb C}
\def\HH{\mathcal H}

\def\AA{\mathcal A}

\def\la {{\lambda}}

\newcommand {\nc}   {\newcommand}
\nc {\be}   {\begin{equation}} \nc {\ee}   {\end{equation}} \nc
{\beq}  {\begin{eqnarray}} \nc {\eeq}  {\end{eqnarray}} \nc {\beqs}
{\begin{eqnarray*}} \nc {\eeqs} {\end{eqnarray*}}
\def\edc{\end{document}}

\providecommand{\abs}[1]{\lvert#1\rvert}

\numberwithin{equation}{section}

\theoremstyle{Thm}
\newtheorem{Thm}{Theorem}[section]
\newtheorem{lem}{Lemma}[section]
\newtheorem{prop}{Proposition}[section]

\newtheorem{rk}{Remark}[section]
\definecolor{carnelian}{rgb}{0.7, 0.11, 0.11}
\definecolor{carmine}{rgb}{0.59, 0.0, 0.09}
\definecolor{burgundy}{rgb}{0.5, 0.0, 0.13}
\definecolor{darkmidnightblue}{rgb}{0.0, 0.2, 0.4}
\definecolor{dimgray}{rgb}{0.75, 0.75, 0.75}
\definecolor{palecarmine}{rgb}{0.69, 0.25, 0.21}
\newcounter{dummy}
\numberwithin{dummy}{section}

\newtheorem{defi}[dummy]{Definition}

\numberwithin{equation}{section}

\def\AA{\mathcal A}
\def\HH{\mathbf{\mathcal H}}

\newcommand{\f}{\mathsf{f}}

\providecommand{\abs}[1]{\lvert#1\rvert}
\begin{document}

	\title[\fontsize{7}{9}\selectfont  ]{Stability and instability results of the Kirchhoff plate equation with delay terms on the boundary control}
\author{Mohammad Akil$^{1}$}
\author{Haidar Badawi$^{2,*}$}
\author{Mohamed Balegh$^{3}$}
\author{Zayd Hajjej$^{4}$}
\address{$^1$ Universit\'e Polytechnique Hauts-de-France, C\'ERAMATHS/DEMAV,
	Valenciennes, France}
\address{$^2$ Department of Mathematics, Lebanese International University, Beirut, Lebanon}
\address{$^3$  Department of Mathematics, College of Science and Arts, Muhayil, King Khalid University, Abha, 61413, Saudi Arabia.}
\address{$^4$ Department of Mathematics, College of Science, King Saud University, P.O. Box 2455, Riyadh 11451, Saudi Arabia.}
\address{$^*$ Corresponding author: haidar.badawi@liu.edu.lb}
\email{Mohammad.Akil@uphf.fr, haidar.badawi@liu.edu.lb, mbalegh@kku.edu.sa,   zhajjej@ksu.edu.sa}
\begin{abstract}
	In this paper, we consider  the Kirchhoff plate equation  with  delay terms on the boundary control are added (see system \eqref{p5-2.1} below). we give some instability examples of system \eqref{p5-2.1} for some choices of delays. Finally, we prove its well-posedness, strong stability without any geometric condition and exponential stability under a multiplier geometric control condition.
\end{abstract}
\maketitle
\pagenumbering{roman}
\maketitle
\tableofcontents
\pagenumbering{arabic}
\setcounter{page}{1}
\section{Introduction}
\noindent 
Let $\Omega \subset \R^2$ be a bounded open set with  boundary $\Gamma$ of class $C^4$ consisting of a clamped part $\Gamma_0 \neq \emptyset$ and a rimmed part $\Gamma_1 \neq \emptyset$ such that $\overline{\Gamma_0}\cap \overline{\Gamma_1}=\emptyset$. We consider  the following  Kirchhoff plate equation with delay terms on the boundary controls:
\begin{equation}\label{p5-2.1}
	\left\{	\begin{array}{lll}
		u_{tt}(x,t)+\Delta^2 u(x,t) =0 \ \ \text{in} \ \ \Omega\times (0,\infty),\vspace{0.15cm}\\
		u(x,t)=\partial_\nu u(x,t)=0 \ \ \text{on} \ \ \Gamma_0 \times (0,\infty),\vspace{0.15cm}\\
		\mathcal{B}_1 u(x,t) = -\beta_1 \partial_\nu u_t(x,t)-\beta_2 \partial_\nu u_t (x,t-\tau_1)\ \ \text{on } \ \ \Gamma_1\times (0,\infty),\vspace{0.15cm}\\
		\mathcal{B}_2 u (x,t) =\gamma_1 u_t(x,t)+ \gamma_2 u_t (x,t-\tau_2) \ \ \text{on} \ \ \Gamma_1 \times (0,\infty),\vspace{0.15cm}\\
		u(x,0)=u_0(x),  \ \ u_t(x,0)=u_1(x) \ \ \text{in} \ \ \Omega,\vspace{0.15cm}\\
	u_t (x,t)=f_0 (x,t) \ \ \text{on} \ \ \Gamma_1 \times (-\tau_1 ,0),\vspace{0.15cm}\\
		\partial_\nu u_t (x,t)=g_0 (x,t) \ \ \text{on} \ \ \Gamma_1 \times (-\tau_2,0).
	\end{array}\right.
\end{equation}
Here  and below, $\beta_1$, $\gamma_1$, $\tau_1$ and $\tau_2$ are positive real numbers,  $\beta_2$ and  $\gamma_2$ are non-zero real numbers, $\nu =(\nu_{1},\nu_{2})$ is the unit outward normal vector along  $\Gamma$, and $\tau =(-\nu_{2},\nu_{1})$ is the unit tangent vector along $\Gamma$. The constant $0<\mu<\frac{1}{2} $ is the Poisson coefficient and the boundary operators $\mathcal{B}_1$ and $\mathcal{B}_2$ are defined   by
$$
\mathcal{B}_1 \mathsf{f}=\Delta \mathsf{f}+(1-\mu) \mathcal{C}_1 \mathsf{f}
$$
and
$$
\mathcal{B}_2 \mathsf{f}=\partial_{\nu}\Delta \mathsf{f}+(1-\mu)\partial_{\tau}\mathcal{C}_2 \mathsf{f},
$$
where
$$
\mathcal{C}_1\mathsf{f}=2\nu_{1}\nu_{2}\mathsf{f}_{x_1x_2} -\nu_{1}^2\mathsf{f}_{x_2x_2}-\nu_{2}^2 \mathsf{f}_{x_1x_1} \ \ \text{and} \ \ \mathcal{C}_2 \mathsf{f}= (\nu_{1}^2-\nu_{2}^2)\mathsf{f}_{x_1x_2}-\nu_{1}\nu_{2}\left(\mathsf{f}_{x_1x_1}-\mathsf{f}_{x_2x_2} \right).
$$
Moreover,  easy computations show that
\begin{equation}\label{p5-JEL}
	\mathcal{C}_1\mathsf{f}=-\partial^2_{\tau}\mathsf{f}-\partial_{\tau} \nu_{2}\mathsf{f}_{x_1} +\partial_{\tau} \nu_{1}\mathsf{f}_{x_2} \ \ \text{and} \ \ \mathcal{C}_2 \mathsf{f}=\partial_{\nu \tau} \mathsf{f} -\partial_{\tau}\nu_{1}\mathsf{f}_{x_1}-\partial_{\tau}\nu_{2}\mathsf{f}_{x_2}.
\end{equation}\\
Now, we reformulate system \eqref{p5-2.1}. For this aim, as in \cite{Nicaise2006},
we introduce the following auxiliary  variables
\begin{equation}
	\begin{array}{lll}
		\displaystyle	z^1 (x,\rho,t):=\partial_\nu u_t (x,t-\rho\tau_1 ),\quad x\in \Gamma_1 ,\,\rho\in(0,1),\, t >0,\vspace{0.15cm}\\
		\displaystyle	z^2 (x,\rho,t):=u_t (x,t-\rho\tau_2 ),\quad x\in \Gamma_1 ,\,\rho\in(0,1),\, t >0.
	\end{array}
\end{equation}
Then, system \eqref{p5-2.1} becomes
\begin{eqnarray}
	\displaystyle  u_{tt}+\Delta^2 u&=&0 \ \  \text{in} \ \   \Omega \times (0,\infty),\label{p5-seq10}\vspace{0.15cm}\\
	\displaystyle u=\partial_\nu u&=&0 \ \ \text{on} \ \ \Gamma_0 \times(0,\infty),\label{p5-seq20}\vspace{0.15cm}\\
	\displaystyle \mathcal{B}_1 u +\beta_1 \partial_\nu u_t+\beta_2 z^1(\cdot,1,t)&=&0 \ \ \text{on} \ \ \Gamma_1 \times (0,\infty)\label{p5-seq30}\vspace{0.15cm},\\
	\displaystyle \mathcal{B}_2u-\gamma_1 u_t -\gamma_2 z^2(\cdot,1,t) &=&0\ \ \text{on} \ \ \Gamma_1 \times (0,\infty)\label{p5-seq40}\vspace{0.15cm},\\
	\displaystyle \tau_1 z^1_t (\cdot,\rho,t)+z^1_\rho (\cdot,\rho,t) &=&0  \ \ \text{on} \ \   \Gamma_1 \times (0,1) \times  (0,\infty),\label{p5-seqd10}\vspace{0.15cm}\\
	\displaystyle \tau_2 z^2_t (\cdot,\rho,t)+z^2_\rho (\cdot,\rho,t)& =&0   \ \ \text{on} \ \   \Gamma_1 \times (0,1) \times  (0,\infty),\label{p5-seqd20}
\end{eqnarray}
with the following initial conditions
\begin{equation}\label{p5-seq70}
	\left\{	\begin{array}{lll}
		u(\cdot,0)=u_0(\cdot), \ \ u_t(\cdot,0)=u_1(\cdot) \ \ \text{in} \ \ \Omega ,\vspace{0.15cm}\\
		z^1 (\cdot,\rho,0)=f_0 (\cdot,-\rho \tau_1) \ \ \text{on} \ \ \Gamma_1 \times (0,1),\vspace{0.15cm}\\
		z^2 (\cdot,\rho,0)=g_0 (\cdot,-\rho \tau_2) \ \ \text{on} \ \ \Gamma_1 \times (0,1).
	\end{array}\right.
\end{equation}
The energy of system \eqref{p5-seq10}-\eqref{p5-seq70} is given by
\begin{equation}\label{p5-E0}
	\begin{array}{lll}
		\displaystyle E(t)=\frac{1}{2} \left\{ a(u,u)+\int_{\Omega}|u_t|^2dx
		+ \tau_1 |\beta_2|\int_{\Gamma_1}\int_{0}^1 \left|z^1 (\cdot,\rho,t)\right|^2 d\rho d\Gamma +\tau_2 |\gamma_2|\int_{\Gamma_1}\int_{0}^1 \left|z^2 (\cdot,\rho,t)\right|^2 d\rho d\Gamma \right\},
	\end{array}
\end{equation}
where the sequilinear form $a:H^2(\Omega)\times H^2(\Omega) \longmapsto \C$ is defined by
\begin{equation}\label{p5-1.3}
	a(\f,\mathsf{g})=\int_{\Omega} \left[\f_{x_1x_1}\overline{\mathsf{g}}_{x_1x_1}+\f_{x_2x_2}\overline{\mathsf{g}}_{x_2x_2}+\mu\left(\f_{x_1x_1} \overline{\mathsf{g}}_{x_2x_2}+\f_{x_2x_2} \overline{\mathsf{g}}_{x_1x_1}\right) +2(1-\mu)\f_{x_1x_2}\overline{\mathsf{g}}_{x_1x_2}\right]dx.
\end{equation}
We first recall the following Green's formula (see \cite{Lagnese}):
\begin{equation}\label{p5-GF}
	a(\f,\mathsf{g})=\int_{\Omega} \Delta^2 \f \overline{\mathsf{g}} dx +\int_{\Gamma}\left( \mathcal{B}_1\f \partial_{\nu}\overline{\mathsf{g}}-\mathcal{B}_2 \f \overline{\mathsf{g}}   \right) d\Gamma, \ \ \forall \f\in H^4 (\Omega ),  \ \mathsf{g} \in H^2 (\Omega).
\end{equation}
For further purposes, we need a weaker version of it. Indeed   as $\mathcal{D}(\overline{\Omega})$ is dense in $E(\Delta^2,L^2(\Omega)):=\\\left\{\f\in H^2(\Omega)  \ |  \ \Delta^2 \f\in L^2(\Omega) \right\}$ equipped with its natural norm, we deduce that $\f\in E(\Delta^2,L^2(\Omega))$ (see Theorem 5.6 in \cite{nicaise1993polygonal}) satisfies $\mathcal{B}_1 \f \in H^{-\frac{1}{2}} (\Gamma)$ and $\mathcal{B}_2 \f \in H^{-\frac{3}{2}} (\Gamma)$ with
\begin{equation}\label{p5-GFw}
	a(\f,\mathsf{g})=\int_{\Omega} \Delta^2 \f \overline{\mathsf{g}} dx +\langle \mathcal{B}_1\f ,\partial_{\nu}\mathsf{g}\rangle_{H^{-\frac{1}{2}}(\Gamma), H^{\frac{1}{2}}(\Gamma)} -\langle \mathcal{B}_2 \f ,\mathsf{g} \rangle_{H^{-\frac{3}{2}}(\Gamma), H^{\frac{3}{2}}(\Gamma)}  , \ \ \forall \mathsf{g} \in H^2 (\Omega).
\end{equation}
{\rm	Similar to \cite{Badawidyn}, for any regular solution $U=(u,u_t,z^1,z^2 )$  of  system \eqref{p5-seq10}-\eqref{p5-seq70}, the energy $E(t)$  satisfies the following estimation }
\begin{equation}\label{p5-3.9}
	\frac{d}{dt} E (t)\leq - (\beta_1 -|\beta_2|)\int_{\Gamma_1} |\partial_\nu u_t|^2 d\Gamma -(\gamma_1-|\gamma_2|)\int_{\Gamma_1}|u_t|^2 d\Gamma.
\end{equation}
Let us recall some previous literature.\\\linebreak

 In 1993, Rao in \cite{rao93} studied the stabilization of the Kirchhoff plate equation with non-linear boundary controls (in the linear case, it corresponds to system \eqref{p5-2.1} with $\beta_2=\gamma_2=0$),   under  a multiplier geometric control condition he established an exponential energy decay rate.\\

 Time delays appear in several applications such as in  physics, chemistry, biology, thermal phenomena not only depending on the present state but also on some past occurrences (see \cite{PhysRevE.57.2150, Kol-mish}). In the last years, the control of partial differential equations with time delays have become popular among scientists, since in many cases time delays induce  some instabilities see \cite{datko1-inst,datko2-inst,datko1985,dreher2009}.\\

  In 2006, Nicaise and Pignotti in \cite{Nicaise2006} studied the multidimensional wave equation  with boundary feedback and a delay term at the boundary, by considering the following system:
 \begin{equation}\label{EQ-Nicaise-2006-1}
 	\left\{
 	\begin{array}{ll}
 		{u}_{tt}(x,t) -  \Delta u (x,t) = 0 \ \ \text{in} \ \   \Omega \times (0,\infty), \nline
 		{u}(x,t)=0\ \ \text{on} \ \   \Gamma_D \times (0,\infty), \nline
 		\frac{\partial u}{\partial \nu} (x,t) = -\mu_1 u_t(x,t)-\mu_2 u_t(x,t-\tau) \ \ \text{on} \ \   \Gamma_N \times (0,\infty), \nline
 		u(x,0)=u_0(x), \ \ u_t(x,0)=u_1(x) \ \ \text{in} \ \ \Omega,\nline
 		\displaystyle{u_t (x,t) = f_0(x,t) } \ \ \text{on} \ \  \Gamma_N \times (-\tau,0),
 	\end{array}
 	\right.
 \end{equation}
where $\mu_1$ and $\mu_2$ are positive real numbers, and  $\Omega$ is an open bounded domain of $\R^n$  with a boundary $\Gamma$ of class $C^2$ and $\Gamma = \Gamma _ D \cup \Gamma_N$, such that $\overline{\Gamma_ D} \cap \overline{\Gamma_N} = \emptyset$. Under the assumption $\mu_2<\mu_1$,  an exponential  decay is achieved. If this assumption does not hold, they found a sequences of delays $\{\tau_k\}_{k}$, $\tau_k\to 0$, for which the corresponding solutions have increasing energy. \\

To the best of our knowledge, it seems  that there  is no result in the existing literature concerning the case of the Kirchhoff plate equation with  boundary controls and time delay. The goal of the present paper is to fill this gap by studying both stability and instability of system  \eqref{p5-2.1}.\\

The outline of this paper is as follows. In section 2, if $|\beta_2|\geq \beta_1$ and $|\gamma_2|\geq \gamma_1$, we give some instability examples of system \eqref{p5-2.1} for some particular choices of delays. In  subsection 3.1 , we prove the well-posedness  of our system. The subsection 3.2 is devoted to establish the strong stability of our system  by following a general criteria of Arendt and Batty. Finally, in the subsection 3.3, under the   \textbf{(MGC)} condition, we show that  system \eqref{p5-2.1} is exponentially stable. \\

Let us finish this introduction with some notation used in the
remainder of the paper:   The usual
norm and semi-norm of the Sobolev space $H^{s}(\Omega)$ ($s> 0$) are denoted by
$\|\cdot\|_{H^s(\Omega)}$ and $|\cdot|_{H^s(\Omega)}$, respectively.
By $A\lesssim B$, we mean that there exists a constant $C>0$ independent of  $A$, $B$ and  a natural parameter $n$ such that  $A\leq C B$.

\section{Instability results}\label{p5-sec3.3} In this section, we will give some instability examples of system \eqref{p5-2.1} in the cases $|\beta_2|\geq \beta_1$ and $|\gamma_2|\geq \gamma_1$. This is achieved by distinguishing between the following cases:
\begin{equation}\tag{$\rm{IS}_1$}\label{p5-is1}
	|\beta_2|=\beta_1 \ \ \text{and} \ \ |\gamma_2|=\gamma_1,
\end{equation}
\begin{equation}\tag{$\rm{IS}_2$}\label{p5-is2}
	|\beta_2|\geq\beta_1 \ \ \text{and} \ \ |\gamma_2|\geq \gamma_1 \ \ \text{and} \ \ |\beta_2|-\beta_1 +|\gamma_2|-\gamma_1>0.
\end{equation}

\begin{Thm}
	{\rm
		If \eqref{p5-is1} or \eqref{p5-is2}  hold, then there exist  sequences of delays and solutions of \eqref{p5-2.1} corresponding to these delays such that their standard energy is constant.

	}
\end{Thm}
\begin{proof}
	We seek for a solution of system \eqref{p5-2.1} in the form
	\begin{equation}\label{p5-3.32is}
		u(x,t)=e^{i\la t}\varphi(x), \ \ \hbox{ with }\la \neq 0.
	\end{equation}
	Inserting \eqref{p5-3.32is} in \eqref{p5-2.1}, we get
	\begin{equation}\label{p5-3.33is}
		\left\{\begin{array}{lll}
			-\la^2 \varphi +\Delta^2 \varphi =0  \ \ \text{in} \ \ \Omega ,\vspace{0.15cm}\\
			\varphi=\partial_\nu \varphi =0 \ \ \text{on} \ \  \Gamma_0,\vspace{0.15cm}\\
			\mathcal{B}_1 \varphi =-i\la (\beta_1 +\beta_2e^{-i\la \tau_1})\partial_\nu \varphi \ \ \text{on} \  \ \Gamma_1 ,\vspace{0.15cm}\\
			\mathcal{B}_2 \varphi =i\la (\gamma_1 +\gamma_2 e^{-i\la \tau_2})\varphi \ \ \text{on} \ \ \Gamma_1.
		\end{array}\right.
	\end{equation}
	Let $\theta \in H^2_{\Gamma_0}(\Omega)$. Multiplying the first equation in \eqref{p5-3.33is} by $\overline{\theta}$, then using Green's formula, we get
	\begin{equation}\label{p5-3.34is}
		-\la^2 \int_\Omega \varphi \overline{\theta}dx+a(\varphi,\theta)+i\la (\beta_1 +\beta_2 e^{-i\la\tau_1})\int_{\Gamma_1 }\partial_\nu \varphi \partial_{\nu}\overline{\theta}d\Gamma +i\la (\gamma_1 +\gamma_2 e^{-i\la\tau_2})\int_{\Gamma_1 } \varphi \overline{\theta}d\Gamma=0,
	\end{equation}
	for all $ \theta \in H^2_{\Gamma_0}(\Omega). $
	Now,  since $|\beta_2|\geq \beta_1$ and $|\gamma_2|\geq \gamma_1$, then we assume that
	\begin{equation}\label{p5-3.35}
		\cos(\la \tau_1 )=-\frac{\beta_1}{\beta_2 } \ \ \text{and }\ \ \cos(\la \tau_2)=-\frac{\gamma_1}{\gamma_2 }.
	\end{equation}
	Thus, we choose
	\begin{equation}\label{p5-3.36}
		\beta_2\sin(\la \tau_1 )= \sqrt{\beta_2^2-\beta_1^2} \ \ \text{and} \ \ \gamma_2\sin(\la\tau_2 )= \sqrt{\gamma_2^2-\gamma_1^2}.
	\end{equation}
	Inserting \eqref{p5-3.35} and \eqref{p5-3.36} in \eqref{p5-3.34is}, we obtain
	\begin{equation}\label{p5-3.37}
		-\la^2 \int_{\Omega} \varphi \overline{\theta} dx+a(\varphi,\theta)+ \la \sqrt{\beta_2^2-\beta_1^2}\int_{\Gamma_1}\partial_\nu \varphi \partial_\nu \overline{\theta}d\Gamma + \la \sqrt{\gamma_2^2-\gamma_1^2}\int_{\Gamma_1} \varphi  \overline{\theta}d\Gamma =0,
	\end{equation}
	for all $\theta \in H^2_{\Gamma_0}(\Omega)$. Now, taking $\theta=\varphi$ in \eqref{p5-3.37}, we obtain
	\begin{equation}\label{p5-3.38*}
		-\la^2 \int_{\Omega} |\varphi |^2 dx+a(\varphi,\varphi)+ \la \sqrt{\beta_2^2-\beta_1^2}\int_{\Gamma_1}|\partial_\nu \varphi|^2 d\Gamma + \la \sqrt{\gamma_2^2-\gamma_1^2}\int_{\Gamma_1} |\varphi|^2  d\Gamma =0.
	\end{equation}
	Without loss of generality, we can assume that
	\begin{equation}\label{p5-3.38}
		\|\varphi\|_{L^2(\Omega)}=1.
	\end{equation}
	Thus, from \eqref{p5-3.38*} and \eqref{p5-3.38}, we get
	\begin{equation}\label{p5-3.39}
		\la^2 -a(\varphi,\varphi) - \la \sqrt{\beta_2^2-\beta_1^2}q_{\nu}(\varphi) - \la \sqrt{\gamma_2^2-\gamma_1^2}q(\varphi)=0,
	\end{equation}
	where
	\begin{equation}
		q(\varphi)=\int_{\Gamma_1} |\varphi|^2 d\Gamma  \ \ \text{and} \ \ q_{\nu}(\varphi)=\int_{\Gamma_1} |\partial_\nu \varphi|^2 d\Gamma.
	\end{equation}
	We define
	$$
	W:=\left\{w\in H^2_{\Gamma_0}(\Omega) \ | \  \|w\|_{L^2(\Omega)}=1 \right\}.
	$$
	Now, we distinguish two cases.\\
	\textbf{Case 1:} If \eqref{p5-is1} holds, then from \eqref{p5-3.39}, we have
	\begin{equation}
		a(\varphi,\varphi)=\la^2.
	\end{equation}
	Let us define
	\begin{equation}
		\la^2:=\min_{w\in W} a(w,w).
	\end{equation}
	Now, if $\varphi$ verifies
	$$
	a(\varphi,\varphi)=\min_{w\in W} a(w,w),
	$$
	then it easy to see that $\varphi$ is a solution of \eqref{p5-3.34is} and consequently \eqref{p5-3.32is} is a solution of \eqref{p5-2.1}. Moreover, from \eqref{p5-3.32is} and \eqref{p5-E0}, we get
	$$
	E(t) =E(0)\geq a(\varphi,\varphi)+\la^2\int_{\Omega} |\varphi|^2 dx =2\la^2>0, \ \ \forall t\geq 0.
	$$
	Thus, the energy of \eqref{p5-2.1} is constant and positive. Further from our  assumptions
	$$
	\cos(\la \tau_1 )=-1, \ \ \sin(\la \tau_1 )=0, \ \ \cos(\la \tau_2)=-1, \ \ \sin(\la \tau_2)=0,
	$$
	system \eqref{p5-3.33is} becomes
	\begin{equation}\label{p5-3.33is*}
		\left\{\begin{array}{lll}
			-\la^2 \varphi +\Delta^2 \varphi =0  \ \ \text{in} \ \ \Omega ,\vspace{0.15cm}\\
			\varphi=\partial_\nu \varphi =0 \ \ \text{on} \ \  \Gamma_0,\vspace{0.15cm}\\
			\mathcal{B}_1 \varphi =0 \ \ \text{on} \  \ \Gamma_1 ,\vspace{0.15cm}\\
			\mathcal{B}_2 \varphi =0 \ \ \text{on} \ \ \Gamma_1.
		\end{array}\right.
	\end{equation}
	So, we can take a sequence $(\la_n)_n$ of positive real numbers defined by
	$$
	\la_n^2=\Lambda_n^2, \ \ n\in \N,
	$$
	where $\Lambda_n^2$, $n\in \N$, are the eigenvalues for the bi-Laplacian operator with the boundary conditions \eqref{p5-3.33is*}$_2$-\eqref{p5-3.33is*}$_{4}$. Then, setting
	$$
	\la_n \tau_1 =(2k+1)\pi, \ \ k\in \N \ \ \text{and} \ \ \la_n \tau_2 =(2l+1)\pi, \ \ l\in \N,
	$$
	we get the following sequences of delays
	$$
	\tau_{1,n,k}=\frac{(2k+1)\pi}{\la_n}, \ \ k, n \in \N \ \ \text{and} \ \ \tau_{2,n,l}=\frac{(2l+1)\pi}{\la_n}, \ \ l,n \in \N,
	$$
	which becomes arbitrarily small (or large) for suitable choices of the indices $n, k, l \in \N$. Therefore, we have found sets of time delays for which system \eqref{p5-2.1} is not asymptotically stable.\\\linebreak
	\textbf{Case 2:} If \eqref{p5-is2} holds, then from \eqref{p5-3.39}, we have
	\begin{equation}
		\lambda =\frac{1}{2}\left[ \sqrt{\beta_2^2-\beta_1^2}q_{\nu}(\varphi)+\sqrt{\gamma_2^2-\gamma_1^2}q(\varphi)  \pm \sqrt{\left(\sqrt{\beta_2^2-\beta_1^2}q_{\nu}(\varphi)+\sqrt{\gamma_2^2-\gamma_1^2}q(\varphi) \right)^2+4a(\varphi,\varphi)}\right].
	\end{equation}
	Let us define
	\begin{equation}\label{p5-3.45}
		\begin{array}{lll}
			\displaystyle \lambda :=\frac{1}{2} \min_{w\in W} \left\{\sqrt{\beta_2^2-\beta_1^2}q_{\nu}(w)+\sqrt{\gamma_2^2-\gamma_1^2}q(w) \right. \vspace{0.15cm}\\\left.
			\hspace{2cm}+ \displaystyle \sqrt{\left(\sqrt{\beta_2^2-\beta_1^2}q_{\nu}(w)+\sqrt{\gamma_2^2-\gamma_1^2}q(w) \right)^2+4a(w,w)}\right\}.
		\end{array}
	\end{equation}
	Let us prove that if the minimum in the right-hand side of \eqref{p5-3.45} is attained at $\varphi$, that is
	\begin{equation}\label{p5-3.46}
		\begin{array}{lll}
			\displaystyle \sqrt{\beta_2^2-\beta_1^2}q_{\nu}(\varphi)+\sqrt{\gamma_2^2-\gamma_1^2}q(\varphi)  + \sqrt{\left(\sqrt{\beta_2^2-\beta_1^2}q_{\nu}(\varphi)+\sqrt{\gamma_2^2-\gamma_1^2}q(\varphi) \right)^2+4a(\varphi,\varphi)}\vspace{0.5cm}\\
			\displaystyle  :=\min_{w\in W} \left\{\sqrt{\beta_2^2-\beta_1^2}q_{\nu}(w)+\sqrt{\gamma_2^2-\gamma_1^2}q(w) + \sqrt{\left(\sqrt{\beta_2^2-\beta_1^2}q_{\nu}(w)+\sqrt{\gamma_2^2-\gamma_1^2}q(w) \right)^2+4a(w,w)}\right\},
		\end{array}
	\end{equation}
	then $\varphi$ is a solution of \eqref{p5-3.37}. For this aim, take for $\varepsilon \in \R$
	\begin{equation}
		w =\varphi +\varepsilon \theta  \ \ \text{with} \ \ \theta \in H^2_{\Gamma_0} (\Omega) \ \ \text{such that } \ \ \int_{\Omega } \varphi \overline{\theta} dx =0.
	\end{equation}
	Thus, we have
	\begin{equation}
		\|w\|^2_{L^2(\Omega )}=\|\varphi\|^2_{L^2(\Omega)}+\varepsilon^2 \|\theta\|^2_{L^2(\Omega)}=1+\varepsilon^2 \|\theta\|^2_{L^2(\Omega )}.
	\end{equation}
	Now, if we define
	\begin{equation}
		\begin{array}{lll}
			\displaystyle 	f(\varepsilon):=\frac{1}{1+\varepsilon^2 \|\theta\|^2_{L^2(\Omega)}}\left(\sqrt{\beta_2^2-\beta_1^2}q_{\nu}(\varphi+\varepsilon\theta)+\sqrt{\gamma_2^2-\gamma_1^2}q(\varphi+\varepsilon\theta)\vspace{0.5cm}\right.\\
			\displaystyle  + \left. \sqrt{\left(\sqrt{\beta_2^2-\beta_1^2}q_{\nu}(\varphi+\varepsilon\theta)+\sqrt{\gamma_2^2-\gamma_1^2}q(\varphi+\varepsilon\theta) \right)^2+4a(\varphi+\varepsilon\theta,\varphi+\varepsilon\theta)}\right)
		\end{array},
	\end{equation}
	thus, from \eqref{p5-3.46}, we get
	$$
	f(\varepsilon)\geq f(0)=\sqrt{\beta_2^2-\beta_1^2}q_{\nu}(\varphi)+\sqrt{\gamma_2^2-\gamma_1^2}q(\varphi)  + \sqrt{\left(\sqrt{\beta_2^2-\beta_1^2}q_{\nu}(\varphi)+\sqrt{\gamma_2^2-\gamma_1^2}q(\varphi) \right)^2+4a(\varphi,\varphi)},
	$$
	which gives
	$$
	f^\prime (0)=0.
	$$
	Consequently, after an easy computation, we obtain
	\begin{equation}\label{p5-3.37*}
		a(\varphi,\theta)+ \la \sqrt{\beta_2^2-\beta_1^2}\int_{\Gamma_1}\partial_\nu \varphi \partial_\nu \overline{\theta}d\Gamma + \la \sqrt{\gamma_2^2-\gamma_1^2}\int_{\Gamma_1} \varphi  \overline{\theta}d\Gamma =0.
	\end{equation}
	Since any function $\tilde{\theta}\in H^2_{\Gamma_0}(\Omega)$ can be decomposed as
	$$
	\tilde{\theta}=\alpha \varphi+\theta \ \ \text{with} \ \ \alpha \in \R \ \ \text{and} \ \ \theta \in H^2_{\Gamma_0}(\Omega) \ \ \text{such that} \ \ \int_{\Omega} \varphi \overline{\theta}dx =0,
	$$
	from \eqref{p5-3.37*} and \eqref{p5-3.38*}, we obtain that $\varphi$ satisfies \eqref{p5-3.37}. Thus, for such $\la >0$
	$$
	\la \tau_1 =\arccos\left(-\frac{\beta_1}{\beta_2}\right)+2k\pi, \ \ k\in \N \ \ \text{and} \ \ \la \tau_2 =\arccos\left(-\frac{\gamma_1}{\gamma_2}\right)+2l\pi, \ \ l\in \N,
	$$ define a sequences of time delays for which \eqref{p5-2.1} is not asymptotically stable.
\end{proof}
\section{Stability results}
In this section, we will prove the wellposedness, strong stability and exponential stability of system \eqref{p5-seq10}-\eqref{p5-seq70}.
For this aim, we make the following assumptions
\begin{equation}\label{p5-H}\tag{\rm{H}}
\beta_1, \gamma_1 >0, \ \  \beta_2,\gamma_2 \in \R^*, \ \ |\beta_2|<\beta_1 \ \ \text{and} \ \ |\gamma_2|<\gamma_1.
\end{equation}
\subsection{Wellposedness of the system}
In this subsection, we will prove the wellposedness of  system \eqref{p5-seq10}-\eqref{p5-seq70}. Under the hypothesis \eqref{p5-H} and from \eqref{p5-3.9}, system \eqref{p5-seq10}-\eqref{p5-seq70} is dissipative in the sense that its energy is non-increasing with respect to time (i.e. $E^{\prime}(t)\leq 0$). Let us define the Hilbert  space $\HH$ by
$$
\HH=H^2_{\Gamma_0}(\Omega)\times L^2 (\Omega)\times \left( L^2 (\Gamma_1 \times (0,1)) \right)^2 ,
$$
where
$$
H^2_{\Gamma_0}(\Omega)=\left\{\f\in H^2(\Omega)\ | \ \f=\partial_\nu \mathsf{f}=0 \ \text{on} \ \Gamma_0 \right\}.
$$
The Hilbert space $\HH$ is equipped with the following inner product

\begin{equation}\label{p5-norm0}
	\begin{array}{lll}
		\displaystyle 	\left(U,U_1\right)_{\HH}=a(u,u_1)+\int_{\Omega} v\overline{v_1}dx
	 +\tau_1 |\beta_2|\int_{\Gamma_1}\int_{0}^1 z^1  \overline{z_1^1}  d\rho d\Gamma +\tau_2 |\gamma_2|\int_{\Gamma_1}\int_{0}^1 z^2 \overline{z_1^2}  d\rho d\Gamma,
	\end{array}
\end{equation}
where  $U=(u,v, z^1,z^2)^{\top} $, $U^1 =(u_1 ,v_1 , z^1_1,z^2_1 )^{\top}\in\HH$.\\

We define the linear unbounded  operator $\AA:D(\AA)\subset \HH\longmapsto \HH$  by:
\begin{equation*}
	D(\AA)=\left\{\begin{array}{ll}\vspace{0.15cm}
		U=(u,v,z^1,z^2)^{\top}  \in \mathsf{D}_{\Gamma_0}(\Delta^2 )\times H^2_{\Gamma_0}(\Omega) \times (L^2(\Gamma_1;H^1(0,1)))^2  \   | \   \vspace{0.15cm}\\
		\mathcal{B}_1 u=-\beta_1 \partial_\nu v-\beta_2 z^1(\cdot,1),  \ \ \mathcal{B}_2 u=\gamma_1 v+\gamma_2 z^2(\cdot,1), \ \ z^1(\cdot,0)=\partial_\nu v, \ \ z^2 (\cdot,0)=v \ \ \text{on} \ \ \Gamma_1
	\end{array}\right\}
\end{equation*}
where
$$
\mathsf{D}_{\Gamma_0}(\Delta^2)= \left\{ \mathsf{f}\in H^2_{\Gamma_0}(\Omega) \ | \ \Delta^2 \mathsf{f} \in L^2(\Omega), \ \mathcal{B}_1\mathsf{f} \in L^2(\Gamma_1), \ \text{and} \  \mathcal{B}_2 \mathsf{f} \in L^2(\Gamma_1) \right\}
$$
and
\begin{equation}\label{p5-opA0}
	\AA\begin{pmatrix}
		u\\v \\ z^1 \\ z^2
	\end{pmatrix}=
	\begin{pmatrix}
		\displaystyle 	v\vspace{0.15cm}\\  \displaystyle  -\Delta^2 u\vspace{0.15cm} \\ \displaystyle  -\frac{1}{\tau_1} z^1_\rho  \vspace{0.15cm}\\ \displaystyle -\frac{1}{\tau_2} z^2_\rho
	\end{pmatrix}, \forall \; U=(u,v,z^1,z^2 )^{\top}\in D(\AA).
\end{equation}
\begin{rk}
	{\rm
	 From  the fact that $2\Re \left(u_{x_1x_1}\overline{u}_{x_2x_2} \right)=|u_{x_1x_1}+u_{x_2x_2}|^2-|u_{x_1x_1}|^2-|u_{x_2x_2}|^2$, we remark that	
\begin{equation}\label{p5-2.20}
\begin{array}{lll}
\displaystyle |u_{x_1x_1}|^2+|u_{x_2x_2}|^2+2\mu \Re \left(u_{x_1x_1} \overline{u}_{x_2x_2}\right) +2(1-\mu)|u_{x_1x_2}|^2\vspace{0.25cm}\\
\quad =\displaystyle  (1-\mu)|u_{x_1x_1}|^2+(1-\mu)|u_{x_2x_2}|^2+\mu |u_{x_1x_1}+u_{x_2x_2}|^2 +2(1-\mu)|u_{x_1x_2}|^2 \geq 0,
\end{array}
\end{equation}
consequently, from \eqref{p5-1.3}, we get
 $$ a(u,u)\geq (1-\mu) |u|_{H^2(\Omega)}.$$
 Hence the sesquilinear form $a$ is coercive on $H^2_{\Gamma_0}(\Omega)$,
 since $\Gamma_0$ is non empty.
On the other hand,  from \eqref{p5-GFw} (see also Lemma 3.1 and Remark 3.1 in \cite{rao93}), we remark that
\begin{equation}\label{p5-2.21}
a(\mathsf{f},\mathsf{g})=\int_{\Omega} \Delta^2 \mathsf{f} \overline{\mathsf{g}} dx +\int_{\Gamma_1} (\mathcal{B}_1 \mathsf{f} \partial_{\nu} \overline{\mathsf{g}}-\mathcal{B}_2 \mathsf{f} \overline{\mathsf{g}})d\Gamma, \ \ \forall\; \mathsf{f}\in \mathsf{D}_{\Gamma_0}(\Delta^2), \  \mathsf{g}\in H^2_{\Gamma_0}(\Omega).
\end{equation}}
\hfill $\square$
\end{rk}
 Now, if $U=(u, u_t ,z^1 ,z^2)^{\top}$ is solution of \eqref{p5-seq10}-\eqref{p5-seq70}
 and is  sufficiently regular, then system \eqref{p5-seq10}-\eqref{p5-seq70} can be written as the following first order evolution equation
\begin{equation}\label{p5-firstevo0}
	U_t =\AA U , \quad U(0)=U_0,
\end{equation}
where  $U_0 =(u_0 ,u_1 ,f_0 (\cdot,-\rho \tau_1),g_0(\cdot,-\rho\tau_2) )^{\top}\in \HH$.
\begin{prop}\label{p5-amdissip0}{\rm
		Under the hypothesis \eqref{p5-H},	the unbounded linear operator $\AA$ is m-dissipative in the energy space $\HH$.} 
\end{prop}
\begin{proof}
For all $U=(u,v,z^1,z^2)^{\top}\in D(\AA)$, from \eqref{p5-norm0} and \eqref{p5-opA0}, we have
	$$
	\begin{array}{lll}
\displaystyle 	\Re \left(\AA U,U\right)_\HH =\Re \left\{a(v,u)-\int_\Omega \Delta^2 u \overline{v} dx -|\beta_2|\int_{\Gamma_1}\int_{0}^1 z^1_{\rho}  \overline{z^1}  d\rho d\Gamma - |\gamma_2|\int_{\Gamma_1}\int_{0}^1 z^2_{\rho} \overline{z^2}  d\rho d\Gamma \right\}.
	\end{array}
	$$
	Using \eqref{p5-2.21} and the fact that $U\in D(\AA)$, we obtain
\begin{eqnarray}\label{p5-dissipative}
	\begin{array}{lll}
\displaystyle 	\Re \left( \AA U,U\right)_\HH =-\beta_1 \int_{\Gamma_1}|\partial_\nu v|^2 d\Gamma-\Re \left\{\beta_2 \int_{\Gamma_1} z^1 (\cdot,1)\partial_\nu\overline{v}d\Gamma \right\}-\gamma_1 \int_{\Gamma_1}|v|^2 d\Gamma-\Re \left\{\gamma_2 \int_{\Gamma_1} z^2 (\cdot,1)\overline{v}d\Gamma \right\}\vspace{0.25cm}\\
\displaystyle  \hspace{2.5cm}-\frac{|\beta_2|}{2}\int_{\Gamma_1}|z^1 (\cdot,1)|^2 d\Gamma+\frac{|\beta_2|}{2}\int_{\Gamma_1}|\partial_\nu v|^2 d\Gamma -\frac{|\gamma_2|}{2}\int_{\Gamma_1}|z^2 (\cdot,1)|^2 d\Gamma+\frac{|\gamma_2|}{2}\int_{\Gamma_1}|v|^2 d\Gamma.
	\end{array}
	\end{eqnarray}
Now, by using Young's inequality, we get
\begin{equation}\label{p5-2.22}
\left\{	\begin{array}{lll}
\displaystyle 	-\Re \left\{\beta_2\int_{\Gamma_1}z^1 (\cdot,1)\partial_\nu\overline{v} d\Gamma \right\}\leq \frac{|\beta_2|}{2}\int_{\Gamma_1}|z^1 (\cdot,1)|^2 d\Gamma +\frac{|\beta_2|}{2}\int_{\Gamma_1} |\partial_\nu v|^2 d\Gamma,\vspace{0.25cm}\\
\displaystyle 	-\Re \left\{\gamma_2\int_{\Gamma_1}z^2 (\cdot,1)\overline{v} d\Gamma \right\}\leq \frac{|\gamma_2|}{2}\int_{\Gamma_1}|z^2 (\cdot,1)|^2 d\Gamma +\frac{|\gamma_2|}{2}\int_{\Gamma_1} |v|^2 d\Gamma.
	\end{array}\right.
\end{equation}
Inserting \eqref{p5-2.22} in \eqref{p5-dissipative} and using the hypothesis \eqref{p5-H}, we obtain
\begin{equation}\label{p5-dissip0}
\displaystyle \Re (\AA U,U)_{\HH}\leq -(\beta_1-|\beta_2|)\int_{\Gamma_1}|\partial_\nu v|^2 d\Gamma -(\gamma_1-|\gamma_2|)\int_{\Gamma_1}|v|^2 d\Gamma \leq 0, \ \ \forall\; U\in D(\AA)
\end{equation}
		which implies that $\AA$ is dissipative. Now, let us prove that  $\AA$ is maximal. For this aim, if $F=(f_1 ,f_2 ,f_3 ,f_4)^{\top}\in \HH$, we look for $U=(u,v,z^1,z^2 )^{\top}\in D(\AA)$ unique solution  of
		 \begin{equation}\label{p5--au=f}
		-\AA U=F.
	\end{equation}
Equivalently, we have the following system
	\begin{eqnarray}
		-v&=&f_1, \label{p5-f1}\\
		\Delta^2 u &=&f_2,\label{p5-f2}\\
		\frac{1}{\tau_1 }z^1_\rho&=& f_3 ,\label{p5-g1}\\
		\frac{1}{\tau_2}z^2_\rho&=& f_4 ,\label{p5-g2}
	\end{eqnarray}
with the following boundary conditions
\begin{equation}\label{p5-bc}
\begin{array}{lll}
	\displaystyle u=\partial_\nu u=0 \ \ \text{on} \ \ \Gamma_0 \ \  \text{and} \ \ \mathcal{B}_1 u=-\beta_1\partial_\nu v-\beta_2z^1(\cdot,1), \ \mathcal{B}_2 u=\gamma_1 v +\gamma_2 z^2(\cdot,1), \ z^1 (\cdot,0)=\partial_\nu v, \ z^2 (\cdot,0)=v  \ \ \text{on}\ \ \Gamma_1.
\end{array}
\end{equation}
From \eqref{p5-f1} and the fact that $F\in \HH$, we get
\begin{equation}\label{p5-v}
v=-f_1 \in H^2_{\Gamma_0}(\Omega).
\end{equation}
From \eqref{p5-g1}, \eqref{p5-g2}, \eqref{p5-bc} and the fact that $F \in \HH$, we obtain
\begin{equation}\label{p5-z1}
z^1_\rho \in L^2 (\Gamma_1\times (0,1)) \ \ \text{and} \ \ z^1 (\cdot,\rho)=\tau_1 \int_0^\rho f_3 (\cdot,s)ds+\partial_\nu v
\end{equation}
and
\begin{equation}\label{p5-z2}
z^2_\rho \in L^2(\Gamma_1\times (0,1)) \ \ \text{and} \ \ z^2 (\cdot,\rho)=\tau_2 \int_0^\rho f_4 (\cdot,s)ds+v.
\end{equation}
Consequently, from \eqref{p5-v},\eqref{p5-z1}, \eqref{p5-z2} and the fact that $f_3 , f_4\in L^2 (\Gamma_1\times (0,1))$, we deduce that
$$
z^1, z^2\in L^2(\Gamma_1;H^1(0,1)).
$$
It follows from \eqref{p5-f2}, \eqref{p5-bc}, \eqref{p5-z1} and \eqref{p5-z2} that
\begin{equation}\label{p5-2.36}
\left\{\begin{array}{lll}
	\displaystyle \Delta^2 u=f_2 \ \ \text{in} \ \ \Omega,\vspace{0.15cm}\\
	\displaystyle u=\partial_\nu u=0 \ \ \text{on} \ \ \Gamma_0,\vspace{0.15cm}\\
\displaystyle \mathcal{B}_1 u=(\beta_1+\beta_2)\partial_\nu f_1 -\tau_1 \beta_2 \int_0^1 f_3 (\cdot,s)ds \ \ \text{on} \ \ \Gamma_1,\vspace{0.15cm}\\
\displaystyle  \mathcal{B}_2 u=-(\gamma_1+\gamma_2) f_1 +\tau_2 \gamma_2 \int_0^1 f_4 (\cdot,s)ds \ \ \text{on} \ \ \Gamma_1.
\end{array}\right.
\end{equation}
Let $\varphi \in H^2_{\Gamma_0}(\Omega)$. Multiplying the first equation in \eqref{p5-2.36} by $\overline{\varphi}$ and integrating over $\Omega$, then using Green's formula, we obtain
\begin{equation}\label{p5-vf}
a(u,\varphi)=l(\varphi), \ \ \forall \varphi \in H^2_{\Gamma_0}(\Omega),
\end{equation}
where
$$
\begin{array}{lll}
\displaystyle l(\varphi)=\int_\Omega f_2 \overline{\varphi}dx +\int_{\Gamma_1}\left((\beta_1+\beta_2)\partial_\nu f_1 -\tau_1 \beta_2 \int_0^1 f_3 (\cdot,s)ds \right)\partial_\nu \overline{\varphi}d\Gamma\vspace{0.25cm}\\\displaystyle \hspace{2cm}+\int_{\Gamma_1 }\left( (\gamma_1+\gamma_2)f_1 -\tau_2 \gamma_2 \int_0^1 f_4 (\cdot,s)ds \right)\overline{\varphi}d\Gamma.
\end{array}
$$
	It is easy to see that,  $a$ is a sesquilinear, continuous and coercive form on $H^2_{\Gamma_0}(\Omega)\times H^2_{\Gamma_0}(\Omega)$ and $l$ is an antilinear and continuous form on $H^2_{\Gamma_0}(\Omega)$. Then, it follows by Lax-Milgram theorem that \eqref{p5-vf} admits a unique solution $u\in H^2_{\Gamma_0}(\Omega) $. By taking the test function
$\varphi\in\mathcal{D} (\Omega)$, we see that the first identity of \eqref{p5-2.36}  holds in the distributional sense, hence $\Delta^2 u \in L^2(\Omega)$. Coming back to \eqref{p5-vf}, and again applying Greens's formula \eqref{p5-GFw}, we find that
$$
\displaystyle \mathcal{B}_1 u=(\beta_1+\beta_2)\partial_\nu f_1 -\tau_1 \beta_2 \int_0^1 f_3 (\cdot,s)ds \ \ \text{on} \ \ \Gamma_1\\
$$
and
$$
\displaystyle  \mathcal{B}_2 u=-(\gamma_1+\gamma_2) f_1 +\tau_2 \gamma_2 \int_0^1 f_4 (\cdot,s)ds \ \ \text{on} \ \ \Gamma_1.
$$
Further since $F \in \HH$, we deduce that $u\in \mathsf{D}_{\Gamma_0}(\Delta^2)$.
 Consequently,  if we define $U=(u,v,z^1 ,z^2 )^{\top}$ with $u\in H^2_{\Gamma_0}(\Omega)$ the unique solution of \eqref{p5-vf}, $v=-f_1$, and  $z^1$ (resp. $z^2$) defined by \eqref{p5-z1} (resp.  \eqref{p5-z2}),  $U$ belongs to $D(\AA)$ is the unique solution of \eqref{p5--au=f}. Then, $\mathcal{A}$ is an isomorphism and since $\rho\left(\mathcal{A} \right)$ is open set of $\mathbb{C}$ (see Theorem 6.7 (Chapter III) in \cite{Kato01}),  we easily get $R(\lambda I -\mathcal{A}) = {\mathcal{H}}$ for a sufficiently small $\lambda>0 $. This, together with the dissipativeness of $\mathcal{A}$, imply that   $D\left(\mathcal{A}\right)$ is dense in ${\mathcal{H}}$   and that $\mathcal{A}$ is m-dissipative in ${\mathcal{H}}$ (see Theorems 4.5, 4.6 in  \cite{Pazy01}). The proof is thus complete.
	\end{proof}\\\linebreak

According to Lumer-Phillips theorem (see \cite{Pazy01}), Proposition \ref{p5-amdissip0} implies that the operator $\AA$ generates a $C_{0}$-semigroup of contractions $e^{t\AA}$ in $\HH$ which gives the well-posedness of \eqref{p5-firstevo0}. Then, we have the following result:
\begin{Thm}{\rm
		For all $U_0 \in \HH$,  system \eqref{p5-firstevo0} admits a unique weak solution $U(t)=e^{t\AA}U_0  \in C^0 (\R^+ ,\HH).
		$ Moreover, if $U_0 \in D(\AA)$, then the system \eqref{p5-firstevo0} admits a unique strong solution $U(t)=e^{t\AA}U_0 \in C^0 (\R^+ ,D(\AA))\cap C^1 (\R^+ ,\HH).$}
\end{Thm}
\subsection{Strong Stability}\label{p5-sec3}
\noindent In this subsection, we will prove the strong stability of  system \eqref{p5-seq10}-\eqref{p5-seq70}. The main result of this section is the following theorem.
\begin{Thm}\label{p5-strongthm20}
	{\rm Under the hypotheses \eqref{p5-H} and \eqref{p5-MGC},	the $C_0-$semigroup of contraction $\left(e^{t\AA}\right)_{t\geq 0}$ is strongly stable in $\HH$; i.e., for all $U_0\in \HH$, the solution of \eqref{p5-firstevo0} satisfies
		$$
		\lim_{t\rightarrow +\infty}\|e^{t\AA}U_0\|_{\HH}=0.
		$$}
\end{Thm}\noindent

 \noindent According to Arendt-Batty \cite{Arendt01}, to prove Theorem \ref{p5-strongthm20}, we need to prove that the operator $\AA$ has no pure imaginary eigenvalues and $\sigma(\AA)\cap i\R $ is countable. The proof of these results is not reduced to the analysis of the point spectrum of $\AA$ on the imaginary axis since its resolvent is not compact. Hence
the proof of Theorem \ref{p5-strongthm20} has been divided into the following two Lemmas.
\begin{lem}\label{p5-ker}
	{\rm For all $\la \in \R$, $i\la I-\AA$ is injective i.e.,
		$$\ker(i\la I-\AA)=\{0\}.$$

	}
\end{lem}
\begin{proof}
	From Proposition \ref{p5-amdissip0}, we have $0\in \rho (\AA)$. We still need to show the result for $\la \in \R^{*}$. For this aim, suppose that $\la\neq0$ and let $U=(u,v,z^1, z^2)^{\top}\in D(\AA)$ be such that
	\begin{equation}\label{p5-AU=ilaU}
		\AA U=i\la U.
	\end{equation}Equivalently, we have the following system

\begin{eqnarray}
	v&=&i\la u,\label{p5-f1k}\\
	-\Delta^2 u &=& i\la v, \label{p5-f2k}\\
	-\frac{1}{\tau_1}z^1_\rho  &=& i\la z^1 ,\label{p5-f5k}\\
			-\frac{1}{\tau_2}z^2_\rho &=& i\la z^2 .\label{p5-f6k}
	\end{eqnarray}
From \eqref{p5-dissip0}, \eqref{p5-AU=ilaU} and \eqref{p5-H}, we get
\begin{equation*}
0=\Re \left(i\la \|U\|^2_\HH \right)=	\Re \left(\AA U,U\right)_\HH\leq  - (\beta_1 -|\beta_2|)\int_{\Gamma_1} |\partial_\nu v|^2 d\Gamma - (\gamma_1 -|\gamma_2|)\int_{\Gamma_1} |v|^2 d\Gamma\leq 0.
\end{equation*}
Thus, we have
\begin{equation}\label{p5-1.21}
\partial_\nu v=v=0 \ \ \text{on }  \  \ \Gamma_1,
\end{equation}
which gives, from \eqref{p5-f1k} and the fact that $\la \neq 0$, that
\begin{equation}\label{p5-3.12}
u=\partial_\nu u=0 \ \ \text{on} \ \ \Gamma_1.
\end{equation}
Using \eqref{p5-f5k}, \eqref{p5-f6k}, \eqref{p5-1.21} and the fact that  $z^1(\cdot,0)=\partial_\nu v$, $z^2(\cdot,0)=v$ on $\Gamma_1$, we obtain
\begin{equation}\label{p5-1.22}
z^1(\cdot,\rho)=\partial_\nu v e^{-i\la \tau_1 \rho}=0 \ \ \text{on} \ \ \Gamma_1\times (0,1),
\end{equation}
\begin{equation}\label{p5-1.22"}
	z^2(\cdot,\rho)=v e^{-i\la \tau_2 \rho}=0 \ \ \text{on} \ \ \Gamma_1\times (0,1).
\end{equation}

Now, from \eqref{p5-1.21}, \eqref{p5-1.22}, \eqref{p5-1.22"} and the fact that $U\in D(\AA)$, we get
\begin{equation}\label{p5-1.24}
	\mathcal{B}_1 u=\Delta u+(1-\mu)\mathcal{C}_1u=0 \ \ \text{on} \ \ \Gamma_1,
\end{equation}
\begin{equation}\label{p5-1.21"}
 \mathcal{B}_2 u=\partial_{\nu}\Delta u+(1-\mu)\partial_{\tau }\mathcal{C}_2u=0 \ \ \text{on} \ \ \Gamma_1.
\end{equation}
Using \eqref{p5-3.12} and the fact that $\nabla u= \partial_{\tau}u \tau +\partial_{\nu}u \nu \ \text{on} \ \Gamma_1 $, we obtain
\begin{equation}\label{p5-1.26}
	  u_{x_1}=u_{x_2}=0 \ \ \text{on} \ \ \Gamma_1.
\end{equation}
Now, from \eqref{p5-JEL}, \eqref{p5-3.12} and \eqref{p5-1.26}, we get
\begin{equation}
	\mathcal{C}_1u=\mathcal{C}_2u=0 \ \ \text{on} \ \ \Gamma_1,
\end{equation}
consequently, from \eqref{p5-1.24} and \eqref{p5-1.21"}, we get
\begin{equation}
	\Delta u =\partial_{\nu} \Delta u =0 \ \ \text{on} \ \ \Gamma_1.
		\end{equation}
Inserting \eqref{p5-f1k} in \eqref{p5-f2k}, we obtain
\begin{equation}\label{p5-1.28}
\left\{\begin{array}{lll}
	\la^2 u-\Delta^2 u=0 \ \ \text{in} \ \ \Omega,\\[0.1in]
	u=\partial_\nu u =0 \ \ \text{on} \ \ \Gamma_0,\\[0.1in]
	u=\partial_{\nu}u=\Delta u=\partial_{\nu }\Delta u=0 \ \ \text{on} \ \ \Gamma_1.
\end{array}\right.
\end{equation}
Holmgren uniqueness theorem (see \cite{LionsHolm}) yields
\begin{equation}\label{p5-1.29}
	u=0 \ \ \text{in} \ \ \Omega.
\end{equation}
Finally, from \eqref{p5-f1k}, \eqref{p5-1.22}, \eqref{p5-1.22"}, and \eqref{p5-1.29}, we get
$$
U=0.
$$
	\end{proof}
\begin{lem}\label{p5-surj}
	{\rm Under the hypothesis $\eqref{p5-H}$, for all $\la \in \R $, we have
		$$R(i\la I-\AA )=\HH.$$
		
	}
\end{lem}
\begin{proof}
	From Proposition \ref{p5-amdissip0}, we have $0\in\rho(\AA)$. We still need to show the result for $\la \in \R^{\star}$. For this aim, for $F=(f_1,f_2,f_3,f_4)^{\top}\in \HH$, we look for $U=(u,v,z^1,z^2)^{\top}\in D(\AA)$ solution of \begin{equation}\label{p5-2.78}
		(	i\la I -\AA)U=F.
	\end{equation}
Equivalently, we have the following system
\begin{eqnarray}
	i\la u-v&=&f_1,\label{p5-f1s}\\
	i\la v+\Delta^2 u&=&f_2,\label{p5-f2s}\\
	i\la z^1 +\frac{1}{\tau_1}z^1_\rho &=&f_3 ,\label{p5-f5s}\\
		i\la z^2 +\frac{1}{\tau_2}z^2_\rho &=&f_4 \label{p5-f6s},
\end{eqnarray}
with the following boundary conditions
\begin{equation}\label{p5-bcs}
	\begin{array}{lll}
		\displaystyle u=\partial_\nu u=0 \ \ \text{on} \ \ \Gamma_0 \ \  \text{and} \ \ \mathcal{B}_1 u=-\beta_1\partial_\nu v-\beta_2 z^1(\cdot, 1), \ \mathcal{B}_2 u=\gamma_1 v +\gamma_2 z^2(\cdot,1), \ z^1 (\cdot,0)=\partial_\nu v , \ z^2 (\cdot,0)=v  \ \ \text{on}\ \ \Gamma_1.
	\end{array}
\end{equation}
From \eqref{p5-f5s}, \eqref{p5-f6s} and \eqref{p5-bcs}, we deduce that
\begin{equation}\label{p5-z1s}
z^1(\cdot,\rho)=\partial_\nu v e^{-i\la \tau_1\rho}+\tau_1 \int_0^\rho f_3(x,s)e^{i\la \tau_1(s-\rho)}ds \ \ \text{on} \ \ \Gamma_1 \times (0,1),
\end{equation}
\begin{equation}\label{p5-z2s}
	z^2(\cdot,\rho)=v e^{-i\la \tau_2\rho}+\tau_2 \int_0^\rho f_4(x,s)e^{i\la \tau_2(s-\rho)}ds \ \ \text{on} \ \ \Gamma_1 \times (0,1).
\end{equation}

It follows from \eqref{p5-f1s}, \eqref{p5-f2s}, \eqref{p5-bcs}, \eqref{p5-z1s} and \eqref{p5-z2s} that
\begin{equation}\label{p5-3.32}
\left\{ \begin{array}{lll}
\displaystyle -\la^2 u+\Delta^2 u =i\la f_1 +f_2 \ \ \text{in } \ \ \Omega,\\[0.1in]
u=\partial_\nu u=0 \ \ \text{on } \ \ \Gamma_0,\\[0.1in]
\mathcal{B}_1 u=-C_{i\la}( \partial_{\nu} u+\frac{i}{\la}\partial_\nu f_1)-F_{i\la}  \ \ \text{on} \ \ \Gamma_1,\\[0.1in]
\mathcal{B}_2 u=D_{i\la}( u+\frac{i}{\la}f_1)+ G_{i\la}  \ \ \text{on} \ \ \Gamma_1,
\end{array}\right.
\end{equation}
where
$$C_{i\la}= i\la(\beta_1+\beta_2e^{-i\la \tau_1}),\;\;F_{i\la}=\beta_2\tau_1 \displaystyle\int_0^1 f_3(x,s)e^{i\la \tau_1(s-1)}ds,$$
and
$$D_{i\la}=i\la(\gamma_1+\gamma_2e^{-i\la\tau_2}),\;\;G_{i\la}=\gamma_2\tau_2 \displaystyle\int_0^1 f_4(x,s)e^{i\la \tau_2(s-1)}ds .$$
Let $\varphi \in H^2_{\Gamma_0}(\Omega)$. Multiplying the first equation in \eqref{p5-3.32} by $\overline{\varphi}$, integrating over $\Omega$, then using Green's formula, we obtain
\begin{equation}\label{p5-3.33}
	b(u,\varphi)=l(\varphi), \ \ \forall \varphi \in \mathbb{V}:= H^2_{\Gamma_0}(\Omega),
\end{equation}
where
$$
b(u,\varphi)=b_1 (u,\varphi)+b_2 (u,\varphi),
$$
with
\begin{equation}\label{p5-3.34}
\left\{\begin{array}{lll}
\displaystyle b_1 (u,\varphi)=a(u,\varphi),\vspace{0.25cm}\\
\displaystyle b_2(u,\varphi)=-\la^2 \int_{\Omega}u\overline{\varphi}dx +C_{i\la} \int_{\Gamma_1}\partial_{\nu} u \partial_{\nu} \overline{\varphi}d\Gamma +D_{i\la }\int_{\Gamma_1}u\overline{\varphi}d\Gamma
\end{array}\right.
\end{equation}
and
\begin{equation}
	l(\varphi)=\int_{\Omega} (i\la f_1+f_2)\overline{\varphi}dx-\int_{\Gamma_1}(\frac{i}{\la} C_{i\la}\partial_\nu f_1+F_{i\la})\partial_{\nu}\overline{\varphi}d\Gamma- \int_{\Gamma_1} (\frac{i}{\la} D_{i\la}+G_{i\la})\overline{\varphi}d\Gamma.
\end{equation}
Let $\mathbb{V}^{\prime}$ be the dual space of $\mathbb{V}$. Let us define the following operators
\begin{equation}
	\begin{array}{lll}
\mathbb{B}: \mathbb{V}& \longmapsto& \mathbb{V}^{\prime}\\
\ \ \quad u &\longmapsto& \mathbb{B}u
\end{array}
 \ \ \text{and} \ \
	\begin{array}{lll}
	\mathbb{B}_i: \mathbb{V}& \longmapsto& \mathbb{V}^{\prime}\\
	\ \ \quad u &\longmapsto& \mathbb{B}_i u
\end{array}, \ \ i\in \{1,2\},
\end{equation}
such that
\begin{equation}\label{p5-3.54}
	\left\{	\begin{array}{lll}
		\displaystyle	(\mathbb{B}u)(\varphi)=b(u,\varphi),\ \ \forall \varphi \in \mathbb{V},\vspace{0.15cm}\\\displaystyle 		\displaystyle	(\mathbb{B}_i u)(\varphi)=b_i(u,\varphi),\ \ \forall \varphi \in \mathbb{V},\ i\in \{1,2\}.
	\end{array}\right.
\end{equation}
We need to prove that the operator $\mathbb{B}$ is an isomorphism. For this aim, we divide the proof into two steps:\\\linebreak
\textbf{Step 1.} In this step, we  prove that the operator $\mathbb{B}_2 $ is compact. For this aim, let us define the following Hilbert space
\begin{equation*}
H^s_{\Gamma_0}(\Omega):=\left\{\varphi \in H^s (\Omega) \ | \ \varphi=\partial_{\nu}\varphi =0 \ \ \text{on} \ \ \Gamma_0 \right\} \ \ \text{with} \ \ s\in \left(\frac{3}{2},2\right).
\end{equation*}
 Now, from \eqref{p5-3.34} and a trace theorem, we get
\begin{equation*}
	\begin{array}{lll}
\displaystyle |b_2(u,\varphi) |\lesssim \|u\|_{L^2(\Omega)}\|\varphi\|_{H^2(\Omega)}+\|\partial_\nu u\|_{L^2(\Gamma_1)}\|\partial_\nu \varphi\|_{L^2(\Gamma_1)}+\| u\|_{L^2(\Gamma_1)}\| \varphi\|_{L^2(\Gamma_1)} \vspace{0.25cm}\\
\displaystyle \hspace{1.5cm}\lesssim \|u\|_{H^s(\Omega)}\|\varphi\|_{H^2(\Omega)},
\end{array}
\end{equation*}
for all $ s\in \left( \frac{3}{2},2\right)$. As $\mathbb{V}$  is compactly embedded into $H^s_{\Gamma_0}(\Omega)$ for any $ s\in \left( \frac{3}{2},2\right)$, $\mathbb{B}_2 $ is indeed a compact operator.
\\\linebreak
This compactness property and the fact that $\mathbb{B}_1$ is an isomorphism imply that the operator $\mathbb{B}=\mathbb{B}_1 +\mathbb{B}_2 $ is a Fredholm operator of index zero. Now, following Fredholm alternative, we simply need to prove that the operator $\mathbb{B}$ is injective to obtain that it is an isomorphism.\\\linebreak
\textbf{Step 2.} In this step, we  prove that the operator $\mathbb{B}$ is injective (i.e. $\ker(\mathbb{B})=\{0\}$). For this aim, let $\mathsf{u}\in \ker(\mathbb{B})$ which gives
\begin{equation*}
	b(\mathsf{u},\varphi)=0,\ \ \forall \varphi\in \mathbb{V}.
\end{equation*}Equivalently, we have
\begin{equation*}
a(\mathsf{u},\varphi)-\la^2 \int_{\Omega}\mathsf{u}\overline{\varphi}dx + C_{i\la} \int_{\Gamma_1}\partial_{\nu} \mathsf{u}\partial_{\nu} \overline{\varphi}d\Gamma +D_{i\la }\int_{\Gamma_1}\mathsf{u}\overline{\varphi}d\Gamma=0, \ \forall \varphi \in \mathbb{V}.
\end{equation*}
Thus, we find that

\begin{equation*}
	\left\{\begin{array}{lll}
		\displaystyle	-\la^2 \mathsf{u} +\Delta^2 \mathsf{u}=0 \ \ \text{in} \ \ \mathcal{D}^\prime (\Omega), \vspace{0.15cm}\\
		\mathsf{u}=\partial_{\nu}\mathsf{u}=0  \ \ \text{on} \ \  \Gamma_0 \vspace{0.15cm}\\
		\mathcal{B}_1 \mathsf{u}= -C_{i\la} \partial_{\nu} \mathsf{u} \ \ \text{on} \ \ \Gamma_1, \vspace{0.15cm}\\ \mathcal{B}_2 \mathsf{u}=D_{i\la} \mathsf{u} \ \ \text{on} \ \ \Gamma_1.\vspace{0.15cm}
	\end{array}\right.
\end{equation*}Therefore, the vector $\mathsf{U}$ defined by
\[
\mathsf{U}=(\mathsf{u},i\la \mathsf{u},i\la e^{-i\la \tau_1 \rho}  \partial_{\nu} \mathsf{u}, i\la e^{-i\la \tau_2 \rho } \mathsf{u})^{\top}\] belongs to $D(\AA )$  and satisfies $$i\la \mathsf{U}-\AA\mathsf{U}=0,
$$
and consequently $\mathsf{U}\in \ker(i\la I-\AA)$. Hence Lemma \ref{p5-ker} yields $\mathsf{U}=0$ and consequently $\mathsf{u}=0$ and $\ker(\mathbb{B})=\{0\}$.
\\\linebreak
Steps 1 and 2 guarantee that the operator $\mathbb{B}$ is isomorphism. Furthermore it is easy to see that the  operator $l$ is an antilinear and continuous form on $\mathbb{V}$. Consequently,  \eqref{p5-3.33} admits a unique solution $u\in \mathbb{V} $. In  \eqref{p5-3.33}, by taking  test functions
$\varphi\in\mathcal{D} (\Omega)$, we see that the first identity of
 \eqref{p5-3.32}  holds in the distributional sense, hence $\Delta^2 u \in L^2(\Omega)$. Coming back to \eqref{p5-3.33}, and again applying Green's formula \eqref{p5-GFw}, we find that
 $$
 \mathcal{B}_1 u=-C_{i\la}( \partial_{\nu} u+\frac{i}{\la}\partial_\nu f_1)-F_{i\la}  \ \ \text{on} \ \ \Gamma_1,
$$
 and
 $$
 \mathcal{B}_2 u=D_{i\la}( u+\frac{i}{\la}f_1)+ G_{i\la}  \ \ \text{on} \ \ \Gamma_1.
 $$
 Further since $u$, $\partial_\nu u$, $f_1$, $\partial_\nu f_1$, $\mathsf{F}_{i\la}$ and $\mathsf{G}_{i\la}$ belong to $L^2(\Gamma_1)$, we deduce that $u\in \mathsf{D}_{\Gamma_0}(\Delta^2)$.
 Consequently, if   $u\in \mathbb{V} $ is the unique solution of \eqref{p5-3.33} and if  we define $z^1$ (resp. $z^2$) by \eqref{p5-z1s}
 (resp. \eqref{p5-z2s}),
 we deduce that
 \[
U=(u,i\la u-f_1,z^1, z^2)^\top\]
 belongs to $D(\AA) $ and  is the unique solution of \eqref{p5-2.78}.
\end{proof}
\\\linebreak
\textbf{Proof of Theorem \ref{p5-strongthm20}.} From Lemma \ref{p5-ker},  the operator $\AA $ has no pure imaginary eigenvalues (i.e. $\sigma_p (\AA)\cap i\R=\emptyset$). 
Moreover, from Lemma \ref{p5-ker} and Lemma \ref{p5-surj},   $i\la I-\AA $ is bijective for all $\lambda\in \mathbb{R}$ and since $\AA$ is closed, we conclude, with the help of the closed graph theorem, that  $i\la I-\AA $ is an isomorphism for all $\lambda\in \mathbb{R}$, hence that
$\sigma(\AA )\cap i\R=\emptyset$.
\xqed{$\square$}
\subsection{Exponential stability}\label{p5-sec3.2}
In this subsection, we will prove the strong stability of  system \eqref{p5-seq10}-\eqref{p5-seq70}.
Let us start up this subsection with the definition of our multiplier geometric control condition.
\begin{defi}\label{p5-defA1}
	{\rm
		We say that the partition $(\Gamma_0, \Gamma_1)$ of the boundary $\Gamma$ satisfies the multiplier geometric control condition \textbf{MGC} if there exists a point $x_0 \in \R^2$ and a positive constant $\delta$ such that
		\begin{equation}\label{p5-MGC}\tag{\rm{GC}}
			h\cdot \nu \geq \delta^{-1} \ \ \text{on} \ \ \Gamma_1 \quad \text{and} \quad h\cdot \nu \leq 0 \ \ \text{on} \ \ \Gamma_0,
		\end{equation}
		where $h(x)=x-x_0$.
		\hfill{$\square$}
		
	}
\end{defi}
\begin{theoreme}\label{p5-exp}{\rm
		Under the hypotheses \eqref{p5-H} and \eqref{p5-MGC}, the $C_0-$semigroup $e^{t\AA}$ is exponentially stable; i.e. there exists constants $M\geq 1$ and $\epsilon>0$ independent of $U_{0}\in \HH$ such that
		$$
		\|e^{t\AA}U_{0}\|_{\HH}\leq Me^{-\epsilon t}\|U_{0}\|_{\HH}, \forall t\geq 0.
		$$}
\end{theoreme}
\begin{proof}  Since $i\R\subset \rho(\AA)$ (see the previous subsection), according to \cite{Huang01} and \cite{pruss01}, to prove  Theorem \ref{p5-exp}, it remains to prove that
\begin{equation}\label{p5-0}
		\limsup_{\la\in \R,\ \abs{\la}\rightarrow \infty}\|\left(i\la I-\AA\right)^{-1}\|_{\mathcal{L}(\HH)}<\infty.
\end{equation}
We will prove condition \eqref{p5-0} by a contradiction argument. For this purpose,
suppose that \eqref{p5-0} is false, then there exists $\left\{(\la_n,U_n:=(u_n,v_n,z^1_n, z^2_n)^{\top})\right\}_{n\geq1}\subset \R^{\ast} \times D(\mathcal{A})$ with
\begin{equation}\label{p5-contra-pol20}
	|\la_n|\to\infty \ \hbox{ as } n\to\infty\quad \text{and}\quad \|U_n\|_{\mathcal{H}}=1,
	\forall n\geq 1,
\end{equation}
such that
\begin{equation}\label{p5-eq0ps0}
 (i\la_n I-\AA )U_n =F_n:=(f_{1,n},f_{2,n},f_{3,n},f_{4,n})^{\top}  \to 0  \quad \text{in}\quad \HH, \ \hbox{ as } n\to\infty.
\end{equation}
For simplicity, we drop the index $n$. Equivalently, from \eqref{p5-eq0ps0}, we have
\begin{eqnarray}
	i\la u-v&=& f_1 \to 0 \ \ \text{in} \ \ H^2_{\Gamma_0}(\Omega),\label{p5-f1p0}\\
	i\la v +\Delta^2 u &=& f_2 \to 0 \ \ \text{in} \ \ L^2(\Omega),\label{p5-f2p0}\\
	i\la z^1+\frac{1}{\tau_1}z^1_\rho &=&  f_3 \to 0 \ \ \text{in} \ \ L^2(\Gamma_1 \times (0,1)),\label{p5-g1p0}\\
	i\la z^2+\frac{1}{\tau_2}z^2_\rho &=&  f_4 \to 0 \ \ \text{in} \ \ L^2(\Gamma_1 \times (0,1)).\label{p5-g2p0}
\end{eqnarray}
Taking the inner product of \eqref{p5-eq0ps0} with $U$ in $\HH$ and using \eqref{p5-dissip0}, we get
 $$
 (\beta_1 -|\beta_2|)\int_{\Gamma_1} |\partial_{\nu}v|^2 d\Gamma +(\gamma_1-|\gamma_2|)\int_{\Gamma_1} |v|^2 d\Gamma \leq -\Re (\AA U,U)_{\HH}=\Re (F,U)_{\HH} \leq  \|F\|_{\HH} \|U\|_{\HH},
 $$
 From the above estimation, \eqref{p5-H} and the fact that $\|F\|_{\HH}=o(1)$ and $\|U\|_{\HH}=1$, we obtain
 \begin{equation}\label{p5-3.26}
 	\int_{\Gamma_1}|\partial_\nu v|^2 d\Gamma =o(1) \ \ \text{and }  \ \ \int_{\Gamma_1}|v|^2d\Gamma =o(1),
 \end{equation}

\begin{lem}\label{p5-lem2}
{\rm
	Under the hypothesis \eqref{p5-H}, the solution $U=(u,v,z^1,z^2)^\top \in D(\AA)$ of \eqref{p5-f1p0}-\eqref{p5-g2p0} satisfies the following estimations
	\begin{equation}\label{p5-3.24}
	\int_{\Gamma_1}\int_0^1 |z^1|^2 d\rho d\Gamma=o(1) \ \ \text{and} \ \ \int_{\Gamma_1} |z^1 (\cdot,1)|^2 d\Gamma =o(1),
	\end{equation}	
\begin{equation}\label{p5-3.29}
	\int_{\Gamma_1}\int_0^1 |z^2|^2 d\rho d\Gamma=o(1) \ \ \text{and} \ \ \int_{\Gamma_1} |z^2 (\cdot,1)|^2 d\Gamma =o(1).
\end{equation}
}
\end{lem}
\begin{proof}
From \eqref{p5-z1s}, Cauchy-Schwarz inequality and the fact that $\rho \in (0,1)$, we get
\begin{equation*}
\begin{array}{lll}
\displaystyle \int_{\Gamma_1}\int_0^1 |z^1|^2 d\rho d\Gamma  \leq 2\int_{\Gamma_1} |\partial_\nu v|^2 d\Gamma +2 \tau_1^2\int_{\Gamma_1}\int_0^1 \left( \int_{0}^\rho |f_3 (\cdot,s)| ds\right)^2 d\rho d\Gamma\vspace{0.25cm}\\
\hspace{3.5cm}\displaystyle \leq 2\int_{\Gamma_1} |\partial_\nu v|^2 d\Gamma +2\tau_1^2\int_{\Gamma_1}\int_0^1 \rho \int_0^\rho |f_3(\cdot,s)|^2 ds d\rho d\Gamma\vspace{0.25cm}\\
\hspace{3.5cm}\displaystyle \leq 2\int_{\Gamma_1} |\partial_\nu v|^2 d\Gamma +2 \tau_1^2\left(\int_0^1 \rho d\rho\right)\int_{\Gamma_1} \int_0^1 |f_3(\cdot,s)|^2 ds  d\Gamma \vspace{0.25cm}\\
\hspace{3.5cm}\displaystyle  = 2\int_{\Gamma_1 }|\partial_\nu v|^2 d\Gamma +\tau_1^2 \int_{\Gamma_1}\int_{0}^1 |f_3 (\cdot,s)|^2 ds d\Gamma.
\end{array}
\end{equation*}
The above inequality, \eqref{p5-3.26} and the fact that $f_3 \to 0 $ in $L^2(\Gamma_1 \times (0,1))$ lead to  the first estimation in \eqref{p5-3.24}. Now, from \eqref{p5-z1s}, we deduce that
$$
z^1(\cdot,1)=\partial_\nu v e^{-i\la \tau_1}+ \tau_1\int_0^1 f_3 (\cdot,s)e^{i\la \tau_1(s-1)}ds \ \ \text{on} \ \ \Gamma_1,
$$
consequently, by using Cauchy-Schwarz inequality, we get
$$
\begin{array}{lll}
\displaystyle \int_{\Gamma_1} |z^1(\cdot,1)|^2 d\Gamma \leq 2 \int_{\Gamma_1 }|\partial_\nu v|^2 d\Gamma +2\tau_1^2\int_{\Gamma_1} \left(\int_0^1 |f_3(\cdot,s)|ds\right)^2d\Gamma\vspace{0.25cm}\\
\hspace{2.5cm}\displaystyle \leq 2\int_{\Gamma_1} |\partial_\nu v|^2 d\Gamma +2\tau_1^2\int_{\Gamma_1} \int_0^1 |f_3(\cdot,s)|^2 dsd\Gamma.
\end{array}
$$
Therefore, from the above inequality, \eqref{p5-3.26} and the fact that $f_3\to 0 $ in $L^2(\Gamma_1\times (0,1))$, we get the second estimation in \eqref{p5-3.24}.
The same argument as before yielding \eqref{p5-3.29}, the proof is  complete.
\end{proof}

Next, from the above estimations, \eqref{p5-3.26} and the fact that $U\in D(\AA)$, we get
\begin{equation}\label{p5-3.25}
\displaystyle \int_{\Gamma_1}|\mathcal{B}_1 u|^2 d\Gamma =o(1)\ \ \text{and} \ \  \int_{\Gamma_1}|\mathcal{B}_2 u|^2 d\Gamma=o(1).
\end{equation}
\begin{lem}\label{p5-lem2111}
{\rm
	Under the hypothesis \eqref{p5-H}, the solution $U=(u,v,z^1,z^2)^\top \in D(\AA)$ of \eqref{p5-f1p0}-\eqref{p5-g2p0} satisfies the following estimations
\begin{equation}\label{p5-3.23}
 		\int_{\Gamma_1} |\partial_\nu u|^2 d\Gamma =o(\la^{-2}) \ \  \text{and} \ \  \int_{\Gamma_1}| u|^2 d\Gamma =o(\la^{-2}).
\end{equation}	
}
\end{lem}
\begin{proof}
Since $U\in D(\AA)$, we have $\mathcal{B}_1 u=-\beta_1\partial_\nu v -\beta_2 z^1(\cdot,1)$ and $\mathcal{B}_2 u= \gamma_1 v +\gamma_2 z^2(\cdot,1)$ \; on $\Gamma_1$.\\

Inserting \eqref{p5-f1p0} in the above equations, we get
$$
i\la \partial_\nu u=-\frac{1}{\beta_1}\mathcal{B}_1 u+\partial_\nu f_1-\frac{\beta_2}{\beta_1}z^1(\cdot, 1) \ \  \text{on} \ \ \Gamma_1 ,
$$
and
$$
i\la u=\frac{1}{\gamma_1}\mathcal{B}_2 u+f_1-\frac{\gamma_2}{\gamma_1}z^2(\cdot, 1) \ \  \text{on} \ \ \Gamma_1.
$$
From the above equations, we deduce that
\begin{equation}\label{est1}
\int_{\Gamma_1} |\la\partial_\nu u|^2 d\Gamma \lesssim \frac{1}{\beta_1^2} \int_{\Gamma_1}|\mathcal{B}_1 u|^2 d\Gamma+\int_{\Gamma_1} |\partial_\nu f_1|^2 d\Gamma+\frac{\beta_2^2}{\beta_1^2}\int_{\Gamma_1}|z^1(\cdot, 1)|^2 d\Gamma,
\end{equation}
and
\begin{equation}\label{est2}
\int_{\Gamma_1} |\la u|^2 d\Gamma \lesssim \frac{1}{\gamma_1^2} \int_{\Gamma_1}|\mathcal{B}_2 u|^2 d\Gamma+\int_{\Gamma_1} | f_1|^2 d\Gamma+\frac{\gamma_2^2}{\gamma_1^2}\int_{\Gamma_1}|z^2(\cdot, 1)|^2 d\Gamma.
\end{equation}
Using a trace theorem and the fact that $a(f_1,f_1)=o(1)$, we get
	$$
	\int_{\Gamma_1} |\partial_{\nu}f_1|^2d\Gamma \lesssim \|f_1\|^2_{H^2(\Omega)}\lesssim a(f_1,f_1)=o(1),
	$$
and
$$
\int_{\Gamma_1} |f_1|^2d\Gamma \lesssim \|f_1\|^2_{H^2(\Omega)}\lesssim a(f_1,f_1)=o(1).
$$
Inserting these estimations in \eqref{est1} and \eqref{est2}, then using Lemma \ref{p5-lem2} and \eqref{p5-3.25}, we get the desired result.
\end{proof}
\begin{lem}\label{p5-lem8}
	{\rm
	Under the hypotheses \eqref{p5-H} and \eqref{p5-MGC}, the solution $U=(u,v,z^1,z^2)^\top \in D(\AA)$ of \eqref{p5-f1p0}-\eqref{p5-g2p0} satisfies the following estimations
	
	\begin{equation}
	\int_{\Omega}|\la u|^2 dx =o(1) \ \ \text{and} \ \ a(u,u)=o(1).
	\end{equation}

}
\end{lem}
\begin{proof}
Inserting \eqref{p5-f1p0} in \eqref{p5-f2p0}, we get
$$
	-\la^2 u+\Delta^2 u=i\la f_1+f_2 \ \ \text{in} \ \ \Omega.
	$$
	Multiplying the above equation by $(h\cdot \nabla \overline{u})$, integrating over $\Omega$, then taking the real part, we obtain
\begin{equation}\label{p5-5.1900}
	\Re \left\{-\la^2 \int_{\Omega} u (h\cdot \nabla \overline{u})dx
	+\int_{\Omega}\Delta^2 u (h\cdot \nabla \overline{u}) dx \right\}
	=\Re \left\{i\la\int_{\Omega}f_1 (h\cdot \nabla \overline{u})dx +\int_{\Omega} f_2 (h\cdot \nabla \overline{u})dx \right\}
\end{equation}
Now, by using Green's formula and the fact that $u=0$ on $\Gamma_0$, then using \eqref{p5-3.23}, we get
\begin{equation}\label{p5-5.2000}
	\Re \left\{-\la^2 \int_{\Omega} u (h\cdot \nabla \overline{u})dx\right\}=\frac{1}{2}\int_{\Omega}|\la u|^2 dx -\frac{1}{2}\int_{\Gamma_1} (h\cdot \nu)|\la u|^2 d\Gamma=\frac{1}{2}\int_\Omega |\la u|^2 dx +o(1).
\end{equation}
Using the fact that $\la^2a(u,u)=O(1)$ and $a(f_1,f_1)=o(1)$, we obtain
$$
\left\{\begin{array}{lll}
|\la|\|\nabla u\|_{L^2(\Omega)}\leq|\la| \|u\|_{H^2(\Omega)}\lesssim |\la|\sqrt{a(u,u)}=O(1),\vspace{0.25cm}\\
\|f_1\|_{L^2(\Omega)} \leq \|f_1\|_{H^2(\Omega)}\lesssim \sqrt {a(f_1,f_1)}=o(1).
\end{array}\right.
$$
Thus, from the above estimations and the fact that $f_2 \to 0$ in $L^2(\Omega)$, we obtain
\begin{equation}\label{p5-5.2100}
\Re \left\{i\la\int_{\Omega}f_1 (h\cdot \nabla \overline{u})dx +\int_{\Omega} f_2 (h\cdot \nabla \overline{u})dx \right\}=o(1).
\end{equation}
	Inserting \eqref{p5-5.2000} in \eqref{p5-5.1900} and using \eqref{p5-5.2100}, we obtain
	\begin{equation}\label{p5-2.12600}
	\frac{1}{2}\int_{\Omega}|\la u|^2 dx= -\Re \left\{\int_{\Omega }\Delta^2 u (h\cdot \nabla \overline{u})dx \right\}+o(1).
	\end{equation}
According to Lemma 5.4 in \cite{Badawidyn}, for all $u\in \mathsf{D}_{\Gamma_0} (\Delta^2)$, we have
\begin{equation}\label{p5-5.18}
	\begin{array}{lll}
		\displaystyle -\Re \left\{\int_\Omega \Delta^2 u (h\cdot \nabla \overline{u})dx \right\} \leq -\frac{1}{2}a(u,u) +\frac{\varepsilon_1 R^2}{2}\int_{\Gamma_1} |\mathcal{B}_2 u|^2 d\Gamma  \vspace{0.25cm}\\
		\qquad \displaystyle +\left( \int_{\Gamma_1 }|\mathcal{B}_1 u|^2 d\Gamma \right)^{\frac{1}{2}} \left(\int_{\Gamma_1}|\partial_{\nu}u|^2 d\Gamma \right)^{\frac{1}{2}}+\frac{R^2 \varepsilon_2}{2}\int_{\Gamma_1}|\mathcal{B}_1u|^2 d\Gamma,
	\end{array}
\end{equation}
where $R=\|h\|_{L^\infty(\Omega)}$ and $\varepsilon_1$, $\varepsilon_2$ are positive constants. Consequently, using \eqref{p5-3.23} and \eqref{p5-3.25}, we obtain
\begin{equation}\label{p5-3.300}
	-\Re \left\{\int_{\Omega }\Delta^2 u (h\cdot \nabla \overline{u})dx \right\}\leq -\frac{1}{2}a(u,u)+o(1).
\end{equation}
Finally, inserting \eqref{p5-3.300} in \eqref{p5-2.12600}, we get
$$
\frac{1}{2}\int_{\Omega }|\la u|^2 dx +\frac{1}{2} a(u,u)=o(1).
$$
The proof is thus complete.
\end{proof}
\linebreak
\textbf{Proof of Theorem \ref{p5-exp}:} From Lemmas \ref{p5-lem2} and \ref{p5-lem8}, we deduce that
$$
\|U\|_\HH=o(1),
$$
which contradicts \eqref{p5-contra-pol20}.
\end{proof}

		


{\bf Acknowledgment}: Mohamed Balegh extends his appreciation to the Deanship of Scientific Research at King Khalid University, Saudi Arabia for funding this work through Small Groups Project under grant number R.G.P.1/169/43.


\end{document}